\documentclass[11pt,leqno]{amsart}
\usepackage{amscd}
\usepackage{pb-diagram}
\usepackage{amssymb}
\usepackage{amsfonts}
\usepackage{latexsym}
\usepackage{verbatim}
\usepackage{amsthm}
\usepackage{amscd}
\usepackage{color}
\numberwithin{equation}{section}
\newcommand{\N}{\nabla}

\theoremstyle{plain}
\newtheorem{theorem}{Theorem}[section]

\newtheorem{lemma}[theorem]{Lemma}
\newtheorem{prop}[theorem]{Proposition}
\newtheorem{cor}[theorem]{Corollary}
\newtheorem{rem}[theorem]{Remark}
\newtheorem{ex}[theorem]{Example}
\newcommand{\ov}[1]{\overline{#1}}

\newcommand{\w}[1]{\wedge}

\newcommand{\e}{\begin{equation}}
\newcommand{\ee}{\end{equation}}

\newcommand\R{{\mathbb R}}

\renewcommand\d{{\partial}}

\newcommand\dd{{\partial\bar{\partial}}}
\newcommand\f{{\varphi}}
\newcommand\g{{\frak{g}}}

\renewcommand\w{{\wedge}}

\subjclass{53C44, 53C10, 53C25, 53C55}
\thanks{This work was supported by by G.N.S.A.G.A. of I.N.d.A.M}
\address{Dipartimento di Ingegneria e Scienze dell'Informazione e Matematica \\ Universit\`a dell'Aquila\\
via Vetoio\\ 67100 L'Aquila\\ Italy}
\email{lucio.bedulli@univaq.it}
\address{Dipartimento di Matematica G. Peano \\ Universit\`a di Torino\\
Via Carlo Alberto 10\\
10123 Torino\\ Italy}
 \email{luigi.vezzoni@unito.it}

\begin{document}
\title{Stability of geometric flows of closed forms}
\author{Lucio Bedulli and Luigi Vezzoni}
\date{\today}

\begin{abstract}
We prove a general result about the stability of geometric flows of \lq\lq closed\rq\rq\, sections of vector bundles on compact manifolds.
Our theorem allows us to prove a stability result for the modified Laplacian coflow in ${\rm G}_2$-geometry introduced by Grigorian in \cite{Gri} and for the balanced flow introduced by the authors in \cite{Wlucione}.  
\end{abstract}
\maketitle



\section{Introduction}
In \cite{Bryant} Bryant introduced a new flow in ${\rm G}_2$-geometry which evolves an initial closed ${\rm G}_{2}$-structure along its Laplacian.  Bryant's {\em Laplacian flow}
is a flow of closed $3$-forms and its well-posedness is not standard since the evolution equation is weakly  parabolic only in the direction of closed forms. The short-time existence of the flow on compact manifolds was proved by Bryant and Xu in \cite{BryantXu} introducing a gauge fixing of the flow called  {\em Laplacian-DeTurck flow}  and then applying Nash-Moser theorem.  

In \cite{L2}  Lotay and Wei proved that in the compact case torsion-free ${\rm G}_2$-structures are stable under the Laplacian flow.  This means that if the initial datum is \lq\lq close enough\rq\rq\, to a torsion free ${\rm G}_2$-structure, the Laplacian flow is defined for any positive time $t$ and converges as $t\to\infty$ in $C^\infty$-topology to a torsion-free ${\rm G}_2$-structure.

Following Bryant and Xu ideas, other similar flows have been introduced in ${\rm G}_2$-geometry.  For instance Karigiannis, McKay and Tsui defined in \cite{K} the {\em Laplacian coflow} which is the \lq\lq dual flow\rq\rq\, to the Laplacian flow since it evolves a closed ${\rm G_2}$ $4$-form along its Laplacian. Although  the Laplacian flow and coflow are similar from the geometric point of view, it turns out that their defining equations are quite different from the analytic point of view and the well-posedness of the Laplacian coflow is still an open problem.
To overcome this technical difficulty,  Grigorian modified in \cite{Gri} the Laplacian coflow by introducing two extra terms, one of which depends on a parameter $A$. In the compact case,  this modification is always well-posed for any choice of $A\in \R$ \cite{Gri}, but it has been shown that the  behaviour of the flow may significantly depend on the choice of $A$ (see \cite{Anna1}).

In \cite{Wlucione} the authors showed that the proof of Bryant and Xu about the well-posedness of the Laplacian-DeTurck flow can be generalized to a quite large family of flows proving a general result which allows us to treat short-time behaviour of Grigorian's modified Lapalcian coflow and of a new flow of balanced metrics in Hermitian geometry.

In the same spirit in the present paper we prove a general result about the stability of a significant class of flows around linearly stable static solutions. Our theorem can be used to re-obtain the Lotay-Wei stability of the Laplacian-DeTurck flow around torsion-free $G_2$-structures (which is a significant part of the main theorem in \cite{L2}). 
As a main application of our theorem we prove the stability of the modified Laplacian coflow when the parameter $A$ is zero. Note that in \cite{Gri} it is suggested to consider only the case $A>0$ and big enough, in order to ensure at least initially that the volume increases.
However the term involving the parameter $A$ does not affect the well-posedness of the flow and our result suggests to consider the case $A=0$ as the best choice for the parameter.

Finally the main result of the present paper  applies to the geometric flow of balanced metrics introduced by the authors in \cite{Wlucione} and yields the stability of the flow around Ricci-flat K\"ahler metrics. 

The paper is organized as follows. In section \ref{2} we give the statement of the main result and we declare some notation we will use in the sequel. Section \ref{coflow} is devoted to the proof of the stability of the modified Laplacian coflow around torsion-free ${\rm G}_2$-structures when the parameter $A$ is zero. The proof is obtained by  mixing the use of our main theorem  with some techniques used in \cite{L2} to prove the stability of the Laplacian flow. In section \ref{3} we prove a stability result involving the balanced flow around Calabi-Yau metrics. In the last section we give the proof of the main theorem.

\bigskip 

\noindent {\bf Acknowledgements}. The authors would like to thank Jason Lotay and Gao Chen for useful discussions. The authors are also grateful to the anonymous referee
for raising important points and helping to considerably improve the presentation of the paper.
\section{Statement of the main result}\label{2}
In this section we describe our setting and give the precise statement of our result. \\
Following the terminology introduced in  \cite{Wlucione} a {\em Hodge system} on a compact Riemannian manifold $(M,g)$ consists of a quadruplet $(E_-,E,D,\Delta_D)$, where $E_-$ and $E$  are vector bundles over $M$ with an assigned metric along their fibers, $D\colon C^{\infty}(M,E_-)\to C^{\infty}(M,E)$ and  $\Delta_D\colon C^{\infty}(M,E)\to C^{\infty}(M,E)$
are differential operators such that
\begin{equation}\label{G}
\psi=DGD^*\psi
\end{equation}
for every $\psi\in {\rm Im}\,D$, where $G$ is the Green operator of $\Delta_D$ and  $D^*$ is the formal adjoint of $D$. The foremost example of Hodge system over $M$ is defined by  $E_-=\Lambda^p$, $E=\Lambda^{p+1}$, $D=d$ and $\Delta_D=dd^*+d^*\!d$ is the standard Laplace operator, on a compact Riemannian manifold. Condition \eqref{G} in this case is a consequence of the standard Hodge theory.
Another interesting example of Hodge system occurs in the study of {\em balanced metrics} in complex geometry 
and it is defined by setting $D=i\partial \bar\partial$ and as $\Delta_D$ the {\em Aeppli Laplacian}
(see the discussion in section \ref{3}). 

Given a compact manifold $M$ with a Hodge system $(E_-,E,D,\Delta_D)$, we consider 
an open fiber subbundle $\mathcal E$ of $E$ and 
a  partial differential operator of order $2m$
$$
Q\colon C^{\infty}(M,\mathcal E)\to C^{\infty}(M,E)\,,
$$
and a linear partial differential operator 
$$
D_+\colon C^{\infty}(M,E)\to C^{\infty}(M,E_+)
$$
such that 
$$
{\rm Im}\,D\subseteq \ker\, D_+ \,,
$$
where $E_+$ is a vector bundle over $M$. Let $\Phi=\ker D_+\cap C^{\infty}(M,\mathcal E)$.  We assume 
\begin{enumerate}
\item[1.] $Q(\Phi)\subset {\rm Im}\,D$;

\vspace{0.2cm}
\item[2.] there exists a smooth family of strongly elliptic  linear partial differential operators $L_\varphi\colon C^{\infty}(M,E)\to C^{\infty}(M,E)$, $\varphi\in C^{\infty}(M,\mathcal E)$,
such that 
$$
Q_{*|\varphi}(\psi)=L_{\varphi}(\psi)
$$ 
for every $\varphi\in \Phi$ and $\psi\in {\rm Im}\,D$;

\vspace{0.2cm}
\item[3.] there exists a smooth family of strongly elliptic  linear partial differential operators $l_\varphi\colon C^{\infty}(M,E_-)\to C^{\infty}(M,E_-)$, $\varphi\in C^{\infty}(M,\mathcal E)$,
such that 
$$
Q_{*|\varphi}(D\theta)=Dl_{\varphi}(\theta)
$$
for every $\varphi\in \Phi$ and $\theta \in C^\infty(M,E_-)$. 

\end{enumerate}

In \cite{Wlucione} the authors proved that for every $\varphi_0\in \Phi$ the evolution problem 
\begin{equation}\label{problem}
\partial_t\varphi_t=Q(\varphi_t)\,,\quad \varphi_{|t=0}=\varphi_0\,,\quad \varphi_t\in U\,,
\end{equation}
is always well-posed, where 
$$
U=\{\varphi_0+D\gamma \,\,:\,\,\gamma\in C^\infty(M,E_-)\}\,\cap C^{\infty}(M,\mathcal E)\,.
$$ 

The main result of the present paper is the following 

\begin{theorem} \label{teoremone}
In the situation described above,  let $\bar \varphi\in\Phi$ be such that 
\begin{enumerate}
\item[1.] $Q(\bar\varphi)=0$;

\vspace{0.1cm}
\item[2.] the restriction to $DC^\infty(M,E_-)$ of $L_{\bar \varphi}$ is symmetric and negative definite with respect to the $L^2$ inner product induced by $g$.
\end{enumerate}
Then for every $\epsilon>0$ there exist $\delta>0$ and $C>0$ such that if 
$$
\|\varphi_0-\bar\varphi\|_{C^{\infty}}<\delta
$$
then \eqref{problem} has a unique long-time solution  $\{\varphi_t\}_{t\in [0,\infty)}$ such that 
$$
\|\varphi_t-\bar \varphi\|_{C^{\infty}}<\epsilon\,, \mbox{ and } \quad \|Q(\varphi_t)\|_{C^\infty}\leq C \|Q(\varphi_0)\|_{L^2} \,{\rm e}^{-\lambda t}
$$
for every $t\in [0,\infty)$, where $\lambda$ is half the first positive eigenvalue of $-L_{\bar\varphi}$.  Moreover, $\varphi_t$ 
converges exponentially fast in $C^{\infty}$ topology to  a $\varphi_\infty \in U$ such that $Q(\varphi_\infty)=0$ as $t\to \infty$. 
\end{theorem}

Whenever we write $\|f\|_{C^{\infty}}<\epsilon$, we mean that $\|f\|_{C^k}<\epsilon$ for every $k\in \mathbb N$.
In the statement above the $C^k$-norms and the $L^2$-norm are with respect to the background metric $g$. 
Conditions 1. and 2. in theorem \ref{teoremone} say that $\bar \varphi$ is a  linearly stable fixed point of the flow. Hence roughly speaking the theorem says that  linearly stable fixed points of the class of flows we are considering  are indeed 
dynamically stable.  

\section{From theorem \ref{teoremone} to the stability of  the  modified Laplacian coflow}\label{coflow}

Let $(M,\varphi_0)$ be a compact manifold with a fixed ${\rm G}_2$-structure. The {\em Laplacian flow} is defined as 
\begin{equation}
\label{G2flow}
\partial_t \varphi_t=\Delta_{\varphi_t}\varphi_t\,,\quad d\varphi_t=0\,,\quad \varphi_t\in C^{\infty}(M,\Lambda^3_+)\,, \quad \varphi_{|t=0}=\varphi_0\,,
\end{equation}
where $\Lambda^3_+$ is the fiber bundle whose sections are ${\rm G}_2$-forms on $M$ and for any $\varphi \in C^\infty(M,\Lambda^3_+)$ the Laplacian operator induced by ${\varphi}$ is denoted by $\Delta_{\varphi}$.
The well-posedness of the flow was proved by Bryant and Xu in \cite{BryantXu} applying Nash-Moser inverse function theorem to  the  gauge fixing of the flow given by the following proposition.

\begin{prop}[Bryant-Xu]\label{BXpre} 
There exists a smooth map $V\colon C^{\infty}(M,\Lambda^3_+)\mathcal \to C^{\infty}(M,TM)$ such that the operator 
$$
Q\colon C^{\infty}(M,\Lambda^3_+)\to\Omega^3(M),\quad Q(\varphi)=\Delta_{\varphi}\varphi+\mathcal{L}_{V(\varphi)}\varphi
$$
satisfies 
$$
Q_{*|{\varphi}}(\sigma)=-\Delta_{\varphi}\sigma+d\Psi(\sigma)
$$
for every closed  ${\rm G}_2$-structure $\varphi$  and $\sigma \in d\Omega^2(M)$, where $\mathcal L$ is the Lie derivative and $\Psi$ is an algebraic linear operator on $\sigma$ with coefficients depending on the
torsion of $\varphi$ in a universal way.
\end{prop}
We briefly recall the definition of the map $V$ since we need it for studying the modified Laplacian coflow.  Let $\nabla^0$ be a fixed background torsion-free connection on $M$. For any $\f \in C^{\infty}(M,\Lambda^3_+)$ let $\nabla^\f$ be the Levi-Civita connection of the metric induced by $\f$.  Let 
$$
T^{\f}=\nabla^{\f}-\nabla^0.
$$
We can locally write $T^{\f}=\frac12 T_{jk}^i\partial_{x^i}\otimes dx^j\circ dx^k$. Then $V$ is locally defined as 
\begin{equation}\label{V}
V(\f)^i=c_1\, g^{pq}T^{i}_{pq}+c_2\,g^{ki}T^{j}_{jk}
\end{equation}
where $c_1$ and $c_2$ are universal constants. 

The \lq\lq modified\rq\rq\,\, Laplacian flow 
\begin{equation}\label{modifiedflow}
\partial_t \varphi_t=\Delta_{\varphi_t}\varphi_t+\mathcal{L}_{V(\varphi_t)}\varphi_t\,,\quad d \varphi_t=0\,,\quad  \varphi_t\in C^{\infty}(M,\Lambda^3_+)\,, \quad  \varphi_{|t=0}=\varphi_0\,,
\end{equation} 
is often called the {\em Laplacian-DeTurck} flow. 

\medskip
In \cite{L2} Lotay and Wei proved the following stability result about Laplacian flow   
\begin{theorem}[Lotay-Wei \cite{L2}]\label{Lotay}
Let $\bar{\varphi}$ be a torsion-free ${\rm G}_2$-structure on a compact $7$-manifold $M$. There exists $\delta>0$ such that for any closed ${\rm G}_2$-structure  $\varphi_0$ cohomologous to $\bar \varphi$  and satisfying $\|\varphi_0-\bar\varphi\|_{C^{\infty}} < \delta$,
the Laplacian flow \eqref{G2flow} with initial value $\varphi_0$ exists for all $t\in [0,\infty)$ and
converges in $C^{\infty}$-topology to $\varphi_\infty	\in {\rm Diff}^0\cdot \bar{\varphi}$ as $t\to \infty$. 
\end{theorem}
In the statement above the $C^k$-norms are meant with respect to the metric induced by $\bar \varphi$.
The proof of Lotay-Wei theorem in \cite{L2} can be subdivided in two steps: in the first step it is proved the stability of the Laplacian-DeTurck
 flow and in the second step it is recovered the stability of the Laplacian flow.   The stability of the Laplacian-DeTurck 
 flow can be deduced from our theorem \ref{teoremone} taking into account lemma 4.2 in \cite{L2}. 
Indeed, according to our setting we put 
$$
E_-=\Lambda^2M\,,\quad E=\Lambda^3M\,,\quad E_+=\Lambda^4M\,\quad \mathcal{E}=\Lambda^3_+
$$
$$
D=d\colon \Omega^2(M)\to \Omega^3(M)\,,\quad D_+=d\colon \Omega^3(M)\to \Omega^4(M)
$$
and we take $\Delta_D\colon \Omega^3(M)\to \Omega^3(M)$ to be the Laplacian induced by a fixed background Riemannian metric. 
Furthermore $\Phi$ is the space of closed ${\rm G}_2$-forms  on $M$ and for $\varphi\in \Phi$ we take 
$$
\begin{aligned}
L_{\varphi}&=-\Delta_{\varphi}+d\Psi\,,\mbox{ on $3$-forms};\\
l_{\varphi}&=-\Delta_{\varphi}+\Psi\,,	\,\,\,\mbox{ on $2$-forms}\,,
\end{aligned}
$$
where $\Psi$ is defined in proposition \ref{BXpre}.
If $\bar \varphi$ is a  torsion free ${\rm G}_2$-structure, then $Q(\bar \varphi)=0$ and the restriction of $L_{\bar \varphi}$ to $d\Omega^2(M)$  is $-\Delta_{\bar \varphi}$. Therefore theorem \ref{teoremone} implies that for every $\epsilon>0$ there exists $\delta>0$ such that if $\varphi_0$ is a closed ${\rm G}_2$-structure cohomologous to $\bar \varphi$ satisfying 
$\|\varphi_0-\bar\varphi\|_{C^{\infty}}<\delta$, then flow \eqref{modifiedflow} with initial value $\varphi_0$  has a long-time solution $\tilde\varphi_t$ defined for $t\in [0,\infty)$ such that $\|\tilde \varphi_t - \bar \varphi\|_{C^\infty} < \epsilon$ and $\tilde\varphi_t$ converges in $C^{\infty}$-topology to some torsion-free ${\rm G}_2$-structure $\tilde \varphi_{\infty}$ in $[\bar \varphi]\in H^3(M,\R)$. 
Now lemma 4.2 of \cite{L2} implies that $\tilde \varphi_\infty=\bar \varphi$ if $\epsilon$ is taken small enough. Indeed lemma 4.2 of \cite{L2} implies that for $\epsilon$ small enough the $L^2$-norm of $\tilde \varphi_t-\bar\varphi$ decays exponentially and consequently $\tilde \varphi_{\infty}=\bar\f$.
The long-time existence of the Laplacian flow 
easily follows. Indeed, let $\phi_t$ be the curve of diffeomorphisms solving 
$$
\partial_t\phi_t=-V(\tilde \varphi_t)_{|\phi_t}\,\quad \phi_{0}={\rm Id}_M\,,
$$
then $\varphi_t:=\phi_t^*(\tilde \varphi_t)$, $t\in [0,\infty)$, is a long-time solution of the Laplacian flow with initial condition $\varphi_{|t=0}=\varphi_0$.  The convergence of the flow in theorem \ref{Lotay} is proved in \cite{L2} by using the Shi-type estimates for the Laplacian flow proved in \cite{L1}.  

\medskip
Next we focus on the Laplacian coflow. 
The {\em Laplacian coflow} is the analogue of the Laplacian flow where the initial ${\rm G}_2$-structure is assumed to be coclosed instead of closed.  
Indeed a ${\rm G}_2$-structure $\varphi$ on a smooth manifold $M$ can be alternatively given by the $4$-form $\psi=*_{\varphi}\varphi$ and the fixed orientation. In \cite{K}  it is defined the Laplacian coflow
$$
\partial_t\psi_t=\Delta_{\psi_t}\psi_t\,,\quad d\psi_t=0\,,\quad \quad \psi_t\in C^{\infty}(M,\Lambda^4_+)\,,\quad \psi_{|t=0}=\psi_0\,,
$$
where here $\Lambda^4_+$ is the fiber bundle whose sections are ${\rm G}_2$-$4$-forms and $\Delta_{\psi_t}$ is the Laplacian operator induced by ${\psi_t}$. 
Unlike the Laplacian flow, there is no known gauge fixing that makes the Laplacian coflow parabolic, for this reason  
in \cite{Gri} Grigorian proposed the following modification
\begin{equation}\label{modified coflow}
\partial_t\psi_t=\Delta_{\psi_t}\psi_t+ 2d((A -{\rm tr}\, T_{\psi_t})*_{\psi_t}\psi_t)\,,\quad d\psi_t=0\,,\quad \quad \psi_t\in C^{\infty}(M,\Lambda^4_+)\,,\quad \psi_{|t=0}=\psi_0\,,
\end{equation}
where $A$ is a constant and for a ${\rm G_2}$ $4$-form $\psi\in C^{\infty}(M,\Lambda^4_+)$ the function ${\rm tr} \,T_{\psi}$ is the trace with respect to the metric $g$ induced by $\psi$ of the torsion tensor 
$$
T_{\psi}(X,Y):=\frac{1}{24}g(\nabla_X *_{\psi}\psi, \iota_Y \psi)
$$ 
for $X,Y$ in $C^{\infty}(M,TM)$, where $\nabla$ is the Levi-Civita connection of $g$. It is not difficult to see that  
$$
{\rm tr} \,T_\psi=\frac14 *_{\psi} (d*_{\psi}\psi\wedge*_{\psi}\psi)\,. 
$$

Under this modification the flow is still not parabolic, but it can be further modified by using a DeTurck trick.  

In this section we revise Grigorian's proof of the well-posedness of the Laplacian coflow in order to show how the flow fits in our setting in the case $A=0$.

We first recall how the space of smooth forms on a ${\rm G}_2$-manifold splits into a direct sum of irreducible modules (we refer to \cite{Bryant} for details). Given a ${\rm G}_2$-manifold $(M,\f)$
the space of $2$-forms and $3$-forms split in irreducible ${\rm G}_2$-modules as 
$$
\Omega^2(M)=\Omega^{2}_{14} (M)\oplus\Omega^2_{7}(M)\,,\quad\quad  \Omega^3(M)=\Omega^{3}_{27}(M)\oplus \Omega^{3}_{7}(M)\oplus \Omega^{3}_{1}(M)
$$ 
where 
$$
\begin{aligned}
\Omega^2_{7}(M)&=\{*_{\f}(\alpha \wedge *_{\f}\f)\,\,:\,\,\alpha\in \Omega^1(M)\}\,,\\
\Omega^{2}_{14}(M)&=\{\alpha\in \Omega^2(M)\,\,:\,\,\alpha\wedge \f=-*_\f\alpha \}
\end{aligned}
$$
and
\begin{eqnarray*}
&\Omega^{3}_{27}(M)&=\{\alpha\in \Omega^3(M)\,\,:\,\,\alpha\wedge \f=\alpha\wedge *_\f\f=0\}\,,\\
&\Omega^{3}_{7}(M)&=\{*_\f(\alpha\wedge \f)\,\,:\,\, \alpha\in \Omega^1(M)\}\,,\\
&\Omega^{3}_{1}(M)&=\{f\,\f\,:\,f\in C^\infty(M)\}\,. 
\end{eqnarray*}
The space of symmetric $2$-tensors $S^2(M)$ on $M$ is isomorphic to  $\Omega^{3}_{1}(M)\oplus \Omega^{3}_{27}(M)$ via the map 
$\mathsf{i}_\f\colon S^2(M)\to \Omega^{3}_{1}(M)\oplus \Omega^{3}_{27}(M)$ locally defined as 
$$
\mathsf{i}_\f(h)=h_r^l\f_{lsk}dx^r\wedge dx^s\wedge dx^k
$$
for every $h=h_{rs}dx^r \circ dx^s$, where $\f_{lsk}$ are the components of $\f$ in the coordinates $\{x^1,\dots,x^7\}$. 
Although the following lemma arises from \cite{Gri}, we prefer to give a proof of it in order to frame the Laplacian coflow in our setting and point out that the vector field needed to apply the DeTurck trick  is the same used in the Laplacian flow. 

\begin{lemma}\label{lemmaG}
Let $Q\colon C^{\infty}(M,\Lambda^4_+)\to \Omega^{4}(M)$ be defined as 
$$
Q(\psi)=\Delta_{\psi}\psi+2d((A-{\rm tr}\, T_{\psi})*_\psi \psi)+\mathcal{L}_{V(*_{\psi}\psi)}\psi\,,
$$
where $V(*_{\psi}\psi)$ is defined  in \eqref{V}. 
Let $\{\psi_t\}_{t\in (-\epsilon,\epsilon)}$ be a smooth curve in $C^{\infty}(M,\Lambda^4_+)$ and 
$$
\f_t=*_{\psi_t}\psi_t\,,
\quad \dot\psi=\partial_{t|t=0}\psi_t \,,\quad \dot\f=\partial_{t|t=0}\f_t\,.
$$ 
Then 
$$
\partial_{t|t=0}Q(\psi_t)=-\Delta_{\psi_0}\dot \psi+2Ad\dot \f+d\Psi(\dot \psi)
$$
where $\Psi$ is  an algebraic linear operator on $\dot \psi$ with coefficients depending on the
torsion of $\psi_0$ in a universal way.
\end{lemma}

\begin{proof}
In this proof we use the same notation as in \cite{Bryant,BryantXu} denoting by $d_{r}^s$ the projection of $d$ onto $\Omega_r^{s}(M)$. We also set $\psi=\psi_0$ and $\varphi=\varphi_0$ to simplify notation. 
 
From \cite{JJDG} it follows that $\dot \f$ and $\dot \psi $ are related as follows 
$$
\dot \f=3f^0\f+*_{\f}(f^1\wedge \f)+f^3\,,\quad \dot \psi=4f^0\psi+f^1\wedge\f-*_{\f}f^3\,,
$$
where $f^0$ is a smooth function, $f^1\in \Omega^1(M)$ and $f^3\in \Omega^3_{27}(M)$.
A direct computation via the formulas in \cite{BryantXu} yields 
$$
\partial_{t|t=0}\Delta_{\psi_t}\psi_t=dp(\dot\psi)\,,\quad \partial_{t|t=0}\mathcal{L}_{V(*_{\psi}\psi)}\psi=dq(\dot\psi)
$$
where 
$$
\begin{aligned}
p(\dot \psi)&=3*_{\f}d(f^0\wedge \f)+*_{\f}df^3+*_\f d*_\f(f^1\wedge \varphi)+{\rm l.o.t.}\,,\\
q(\dot \psi)&=5 *_\f(df^0\wedge \f)+*_\f(d^{27}_7 f^3\wedge \f)+{\rm l.o.t.}\,,
\end{aligned}
$$
and by \lq\lq${\rm l.o.t.}$\rq\rq\, we mean \lq\lq lower order terms\rq\rq. 

Moreover if we denote by $\pi_7(d\dot \psi)$ the component of $d\dot \psi$ in $*_{\f}\Omega^2_{7}(M)=\{\alpha \wedge \psi\,\,:\,\,\alpha\in \Omega^1(M)\}$ we have 
$$
\pi_7(d\dot \psi)=4df^0\wedge \psi+\frac13d^{27}_7f^3\wedge \psi+\frac23 d^7_7f^1\wedge \psi+{\rm l.o.t.}
$$
and from $d\dot\psi=0$ we deduce  
$$
d^{27}_7f^3=-12 df^0-2d_7^7f^1+{\rm l.o.t.}
$$
which implies 
$$
q(\dot \psi)=-7 *_{\f}(df^0\wedge \f)-2*_\f(d^{7}_7 f^1\wedge \f)+{\rm l.o.t.}
$$
Therefore 
$$
p(\dot \psi)+q(\dot \psi)+*_\f d*_\f \dot \psi=2*_\f d*_\f (f^1\wedge \f)-2*_\f(d^{7}_7 f^1\wedge \f)+{\rm l.o.t.}
$$
Now 
$$
*_\f d*_\f (f^1\wedge \f)=\frac47 d^7_1f^1\f+\frac12 *_\f(d_7^7f^1\wedge \f)+ d_{27}^7f^1
$$
and so 
$$
\begin{aligned}
p(\dot \psi)+q(\dot \psi)+*_\f d*_\f \dot \psi&\,=2*_\f d*_\f (f^1\wedge \f)-2*_\f(d^{7}_7 f^1\wedge \f)+{\rm l.o.t.}\\
&\,= \frac87 d^7_1f^1\f+*_\f(d_7^7f^1\wedge \f)+2d_{27}^7f^1-2*_\f(d^{7}_7 f^1\wedge \f)+{\rm l.o.t.}\\
&\,= \frac87 d^7_1f^1\f-*_\f(d_7^7f^1\wedge \f)+2d_{27}^7f^1+{\rm l.o.t.}
\end{aligned}
$$
From
$$
d(*_\f(f^1\wedge \psi))=-\frac37 d^7_1f^1\f-\frac12 *_\f (d_7^7f^1\wedge \f)+d_{27}^7f^1
$$
we deduce 
$$
p(\dot \psi)+q(\dot \psi)+*_\f d*_\f \dot \psi=2d(*_\f(f^1\wedge \psi))+2 d^7_1f^1\f+{\rm l.o.t.}
$$
Finally 
$$
{\rm Tr}\, T_{\psi}=\frac14 *_\f (d\f\wedge\f)=d_1^7f^1+{\rm l.o.t.}
$$
and consequently  
$$
p(\dot \psi)+q(\dot \psi)+2(A-{\rm Tr}\, T_{\psi})\f+*_\f d*_\f \dot \psi=2d(*_\f(f^1\wedge \psi))+2A\f+{\rm l.o.t.}
$$
Therefore 
$$
\partial_{t|t=0}Q(\psi_t)=dp(\dot \psi)+dq(\dot \psi)+2d(A-{\rm Tr}\, T_{\psi})=-\Delta_{\psi}\dot \psi+2Ad\dot \f+{\rm l.o.t.}
$$
and the claim follows. 
\end{proof}


Now we note that torsion-free $\rm G_2$-structures are critical points of the functional $Q$ regardless of the value of $A$.
We concentrate on the case $A=0$. 

The main result of this section is the following  
\begin{theorem}\label{coflow}
Let $\bar \psi\in C^\infty(M,\Lambda^4_+)$ be a torsion-free ${\rm G}_2$-structure on a compact 7-manifold $M$. 
There exists $\delta>0$ such that if $\psi_0 \in C^\infty(M,\Lambda^4_+)$ is closed and satisfies 
$$
\|\psi_0-\bar\psi\|_{C^{\infty}}<\delta\,,\quad [\psi_0]=[\bar \psi]\,,
$$ 
then the evolution equation 
\begin{equation}\label{A=0}
\partial_t\psi_t=\Delta_{\psi_t}\psi_t-2d(({\rm tr}\, T_t)*_{\psi_t}\psi_t)\,,\quad \psi_{|t=0}=\psi_0\,,
\end{equation}
has a unique long-time solution $\{\psi_t\}_{t\in [0,\infty)}$ which converges in $C^{\infty}$-topology to $\psi_\infty	\in {\rm Diff}^0\cdot \bar{\psi}$ as $t\to \infty$. 
\end{theorem}

%
%
%


In the statement above and in the following proof the $C^k$-norms and the $L^2$-norm, where not specified, are meant with respect to the metric $\bar g$ induced by $\bar \psi$.

\begin{proof}
Our approach mixes the use of theorem \ref{teoremone} with some techniques used in \cite{L2} to prove the stability of the Laplacian flow. 
We first apply theorem \ref{teoremone} to show the stability of the gauge fixing of the flow and then
use Shi-type estimates in \cite{shi-type} to recover the stability of the original flow. The proof is subdivided in the following three steps: 
\begin{enumerate}
\item[1.] We  prove that for $\delta$ small enough a gauge fixing to \eqref{A=0}
has a long-time solution $\tilde\psi_t$ which converges in $C^{\infty}$ topology to a torsion-free ${\rm G}_2$-structure $\tilde \psi_{\infty}\in [\bar\psi]$ and stays $C^\infty$-close to $\bar \psi$.

\vspace{0.1cm}
\item[2.] We show that for a suitable choice of $\delta$,  $\tilde\psi_t$ converges in $L^2$-norm to  $\bar\psi$, which implies that $\tilde\psi_{\infty}=\bar\psi$;

\vspace{0.1cm}
\item[3.] We recover the stability of the original flow. 
\end{enumerate}

The proof of steps $2$ and $3$ are close to the case of the Laplacian flow. 

\medskip 
\noindent 
{\em Step $1.$} If we choose as background metric the metric $\bar g$ induced by the torsion free ${\rm G}_2$-structure $\bar \psi$, then lemma \ref{lemmaG} together with theorem \ref{teoremone} implies that for every $\epsilon>0$ there exists $\delta >0$ and $\kappa>0$ such that if $\psi_0 \in C^{\infty}(M,\Lambda^4_+)$ is closed and satisfies 
$$
\|\psi_0-\bar\psi\|_{C^{\infty}}<\delta\,,\quad [\psi_0]=[\bar \psi]\,,
$$ 
then the evolution problem 
\begin{equation}\label{gauge fixing A=0}
\partial_t \tilde \psi_t=\Delta_{\tilde\psi_t}\tilde\psi_t-2d(({\rm tr}\, \tilde T_t)*_{\tilde\psi_t} \tilde\psi_t)+\mathcal{L}_{V(\tilde\psi_t)}\tilde\psi_t\,,\quad \tilde\psi_{|t=0}=\psi_0
\end{equation}
has a long-time solution $\{\tilde \psi_t\}_{t\in [0,\infty)}$ such that 
$$
\|\tilde\psi_t-\bar\psi\|_{C^{\infty}}<	\epsilon\,, \quad \mbox{for every } t \in [0,+\infty)
$$
and 
\begin{equation}\label{exp}
\|\Delta_{\tilde \psi_t}\tilde \psi_t-2d(({\rm tr}\, \tilde T_t)*_{\tilde \psi_t}\tilde \psi_t)+\mathcal{L}_{V(\tilde\psi_t)}\tilde\psi_t\|_{C^{\infty}} \leq \kappa {\rm e}^{-\lambda t}
\end{equation}
for every $t\in[0,\infty)$, where $\lambda$ is half the first positive eigenvalue of $\Delta_{\bar \psi}$. Furthermore, $\tilde \psi_t$ converges exponentially fast to a 
torsion-free ${\rm G}_2$-structure 
$\tilde\psi_{\infty}\in[\bar\psi]$. 

\medskip 
\noindent 
{\em Step $2.$}
We show that if we choose $\epsilon$ small enough in the previous step, then the $L^2$-norm 
of $\theta_t=\tilde\psi_t-\bar\psi$ decays exponentially. Let $Q\colon C^{\infty}(M,\Lambda^4_+)\to \Omega^{4}(M)$ be defined as 
$$
Q(\psi)=\Delta_{\psi}\psi-2d(({\rm tr}\, T_{\psi})*_\psi \psi)+\mathcal{L}_{V(*_{\psi}\psi)}\psi\,.
$$

Since $Q$ sends closed forms to exact forms we can write
\begin{equation}
\label{ev_estimate}
\partial_t\theta_t=Q(\tilde \psi_t)=Q_{*|\bar\psi}(\theta_t)+dF(\tilde\psi_t)=-\Delta_{\bar\psi} \tilde\psi_t+dF(\theta_t)\,,
\end{equation}
where we used also lemma \ref{lemmaG}.
Next we consider $\Theta(\psi)=*_{\psi}\psi$ and write 
$$
\Theta(\tilde \psi_t)=*_{\bar \psi} \bar \psi +S_1(\theta_t)+S_2(\theta_t)
$$
where $S_1(\theta_t)=\Theta_{*|\bar\psi}(\theta_t)$.  
Arguing as in \cite{JJDG} 
we can observe that both $dS_1(\theta_t)$ and $dS_2(\theta_t)$
are dominated by $\theta_t\bar\nabla \theta_t$ for $\|\theta_t\|_{C^\infty_{\bar g}}$ small,
where $\bar \nabla$ is the Levi-Civita connection of $\bar g$. 
Note also that $dS_1(\theta_t)$ is linear in $\bar \nabla \theta_t$. \\

Let $Q^1(\psi):=\Delta_{\psi}\psi$. Then 
$$
Q^1_{*|\bar\psi}(\theta_t)=d*_{\bar \psi}d S_1(\theta_t)
$$
and 
$$
\begin{aligned}
Q^1(\tilde\psi_t)-Q^1(\bar\psi)-Q^1_{*|\bar\psi}(\theta_t)=&\,d*_{\tilde\psi_t}d *_{\tilde \psi_t} \tilde \psi_t - d*_{\bar \psi}d S_1(\theta_t)\\
=&\,d*_{\tilde\psi_t}d *_{\tilde \psi_t} {\tilde \psi_t} - d*_{\bar\psi}d *_{\tilde \psi_t}  {\tilde \psi_t} + d*_{\bar\psi}d *_{\tilde \psi_t} {\tilde \psi_t} - d*_{\bar \psi}d S_1(\theta_t)\\
=&\,d(*_{\tilde\psi_t}-*_{\bar\psi})d *_{\tilde \psi_t} {\tilde \psi_t} + d*_{\bar\psi}d *_{\tilde \psi_t} {\tilde \psi_t} - d*_{\bar \psi}d S_1(\theta_t)\\
=&\,d(*_{\tilde\psi_t}-*_{\bar\psi})d *_{\tilde \psi_t} {\tilde \psi_t} + d*_{\bar\psi}d S_2(\theta_t)\,.\\
\end{aligned}
$$
Thus we can write $\Delta_{\tilde \psi_t}\tilde \psi_t = Q^1_{*|\bar\psi}(\theta_t) + dF^1(\theta_t)$ with $F^1(\theta_t)$ 
dominated by $\theta_t\bar\nabla \theta_t$, for $\|\theta_t\|_{C^\infty_{\bar g}}$ small, since $(*_{\tilde\psi_t}-*_{\bar \psi})$ is a $0^{th}$-order operator  depending  on $\theta_t$ polynomially. 

Now let 
$q^2(\psi)=(*_{\psi}(d^*_\psi \psi\wedge\psi))*_\psi\psi$ and
$Q^2(\psi)=dq^2(\psi)$, then
$$
q^2_{*|\bar\psi}(\theta_t)=(*_{\bar \psi}(*_{\bar \psi} dS_1(\theta_t)\wedge\bar \psi))*_{\bar \psi}\bar\psi\,.
$$
Moreover
$$
\begin{aligned}
&\,q^2(\tilde\psi_t)-q^2(\bar\psi)-q^2_{*|\bar\psi}(\theta_t)=
\,(*_{\tilde\psi_t}(d^*_{\tilde\psi_t} \tilde\psi_t\wedge\tilde\psi_t))*_{\tilde\psi_t}\tilde\psi_t
-(*_{\bar \psi}(*_{\bar \psi} dS_1(\theta_t)\wedge\bar \psi))*_{\bar\psi}\bar\psi\,\\
& = \,((*_{\tilde\psi_t}-*_{\bar\psi})(d^*_{\tilde\psi_t} \tilde\psi_t\wedge(\tilde\psi_t-\bar\psi)))(*_{\tilde\psi_t}\tilde\psi_t-*_{\bar\psi}\bar\psi)
+(*_{\tilde\psi_t}(d^*_{\tilde\psi_t} \tilde\psi_t\wedge(\tilde\psi_t-\bar\psi)))*_{\bar\psi}\bar\psi \\
& +(*_{\tilde\psi_t}(d^*_{\tilde\psi_t} \tilde\psi_t\wedge\bar\psi))*_{\bar\psi}\bar\psi 
 +((*_{\tilde\psi_t}-*_{\bar\psi})(d^*_{\tilde\psi_t} \tilde\psi_t\wedge \bar\psi))(*_{\tilde\psi_t}\tilde\psi_t-*_{\bar\psi}\bar\psi) 
+(*_{\bar\psi}(d^*_{\tilde\psi_t} \tilde\psi_t\wedge \bar\psi))(*_{\tilde\psi_t}\tilde\psi_t-*_{\bar\psi}\bar\psi) \\
& +(*_{\bar\psi}(d^*_{\tilde\psi_t} \tilde\psi_t\wedge(\tilde\psi_t-\bar\psi)))(*_{\tilde\psi_t}\tilde\psi_t-*_{\bar\psi}\bar\psi) 
 -(*_{\bar \psi}(*_{\bar \psi} dS_1(\theta_t)\wedge\bar \psi))*_{\bar\psi}\bar\psi\,,\\
\end{aligned}
$$
and consequently we have that $Q^2(\tilde \psi_t) = Q^2_{*|\bar\psi}(\theta_t) + dF^2(\theta_t)$, with $F^2(\theta_t)$ dominated by 
$\theta_t\bar \nabla\theta_t$, for $\|\theta_t\|_{C^\infty_{\bar g}}$ small. 

Finally, if $Q^3(\psi)=\mathcal{L}_{V(*_{\psi}\psi)}\psi$, then \cite[formula (4.13)]{L2} (once regarded on $4$-forms) yields 
$$
Q^3(\tilde\psi_t) = Q^3_{*|\bar\psi}(\theta_t) + dF^3(\theta_t)\,,
$$
with $F^3(\theta_t)$ again dominated by $\theta_t\bar \nabla\theta_t$, for $\|\theta_t\|_{C^\infty_{\bar g}}$ small. 
Therefore in \eqref{ev_estimate} we can put $F(\theta_t)=F^1(\theta_t)-\frac{7}{2}F^2(\theta_t)+F^3(\theta_t)$ 
and we have that there exists a constant $C' >0$ such that the pointwise estimate $|F(\theta_t)|_{\bar g} \leq C' |\theta_t|_{\bar g} |\bar \nabla\theta_t|_{\bar g}$ holds for $\|\theta_t\|_{C^\infty_{\bar g}}$ small. Now 
$$
\begin{aligned}
\frac{d}{dt}\|\theta_t\|^2_{L^2}=&\,\frac{d}{dt}\int_M|\theta_t|_{\bar g}^2\,{\rm Vol}_{\bar g}
=2\int_M\bar g(\theta_t,\partial_t\theta_t)\,{\rm Vol}_{\bar g}
=2\int_M \bar g(\theta_t,-\Delta_{\bar{\psi}}\theta_t+dF(\theta_t))\,{\rm Vol}_{\bar g}\\
=&\,-2\|d^*\theta_t\|^2_{L^2}+2\int_M\bar g(d^*\theta_t,F(\theta_t))\,{\rm Vol}_{\bar g}
\leq-2\|d^*\theta_t\|^2_{L^2}+2\int_M|d^*\theta_t|_{\bar g} |F(\theta_t)|_{\bar g}\,{\rm Vol}_{\bar g}\\
\leq&	\,-2\|d^*\theta_t\|^2_{L^2}+2C'\int_M|d^*\theta_t|_{\bar g} |\theta_t|_{\bar g}|\bar \nabla\theta_t|_{\bar g}\,{\rm Vol}_{\bar g}
\leq -2\|d^*\theta_t\|^2_{L^2}+2C'\epsilon \int_M|d^*\theta_t|_{\bar g} |\bar \nabla\theta_t|_{\bar g}\,{\rm Vol}_{\bar g}	\\
\leq&	\,-2\|d^*\theta_t\|^2_{L^2}+C'\epsilon (\|d^*\theta_t\|^2_{L^2}+\|\bar \nabla\theta_t\|^2_{L^2})\,.
\end{aligned}
$$ 
Weitzenb\"ock formula yields that there exists a constant $C''>0$ depending only on bounds of the curvature of $\bar g$ such that
$$
\|\bar\nabla\theta_t\|^2_{L^2}\leq \|d^*\theta_t\|^2_{L^2}+C''\|\theta_t\|^2_{L^2}\,.
$$ 
Since
$$
\|d^*\theta_t\|^2_{L^2}\geq 2 \lambda\|\theta_t\|^2_{L^2}
$$
we get 
$$
\frac{d}{dt}\|\theta_t\|^2_{L^2} \leq 4\lambda(C'\epsilon-1)\|\theta_t\|^2_{L^2}+C'C''\epsilon \|\theta_t\|^2_{L^2}=(-4\lambda+C'''\epsilon)\|\theta_t\|^2_{L^2}\,,
$$
with $C'''=C'(4\lambda +C'')$.
Therefore for $\epsilon$ small enough, Gronwall lemma implies  that $\|\tilde \psi_t-\bar\psi\|_{L^2}$ decays exponentially.

\medskip 
\noindent {\em Step $3$.}
We recover from $\tilde \psi_t$ a long-time solution  $\psi_t$ to \eqref{A=0} and we show that $\psi_t$ converges 
 exponentially fast to a torsion-free ${\rm G}_2$-structure in $C^{\infty}$-topology.  Let 
$\{\phi_t\}$ be the family of diffeomorphisms solving 
$$ 
\partial_t\phi_t=-V(*_{\tilde\psi_t}\tilde\psi_t)_{|\phi_t}\,,\quad \phi_{|t=0}={\rm Id}\,,
$$
and correspondingly let us set
$$
\psi_t = \phi_t^* \tilde \psi_t\,.
$$
From the convergence of $\tilde \psi_t$ in $C^{\infty}_{\bar g}$ topology it follows that $\phi_t$ converges to a  limit map $\phi_\infty$. 
Arguing as in \cite{L2} we get that $\phi_{\infty}$ is in fact a diffeomorphism. Indeed, if $X$  is a vector field on $M$,
we have 
$$
\frac12 \d_t |\phi_{t*}(X)|_{\bar g}^2=\bar g\left( \d_t\ \phi_{t*}(X), \phi_{t*}(X)\right)\geq -\left| \d_t \phi_{t*}(X)\right|_{\bar g}\, 
\left|\phi_{t*}(X)\right|_{\bar g}\geq -\,\|V(*_{\tilde\psi_t}\tilde\psi_t )\|_{C^{1}}\left|\phi_{t*}(X)\right|^2_{\bar g}\,.
$$ 
Hence
$$
\d_t \log\,|\phi_{t*}(X)|_{\bar g}\geq -\,\|V(*_{\tilde\psi_t}\tilde\psi_t )\|_{C^{1}}
$$
and integrating we deduce 
$$
|\phi_{t*}(X)|_{\bar g}\geq |X|_{\bar g}\, {\rm e}^{- \,\int_{0}^t\|V(*_{\tilde\psi_s}\tilde\psi_s )\|_{C^{1}}\,ds}\,.
$$

Since $\psi_t$ converges exponentially to $\bar \psi$ in $C^{\infty}$ topology, 
we have that $\|V(*_{\tilde\psi_t}\tilde\psi_t )\|_{C^{1}}$ decays  exponentially so that 
\begin{equation}\label{<>}
|\phi_{t*}(X)|_{\bar g}\geq C\,|X|_{\bar g}\,,
\end{equation}
where $C$  is a positive  constant which does not depend on $X$ and $t$.
This last inequality holds true for $\phi_{\infty}$ and it follows that $\phi_\infty$ is a local diffeomorphism homotopic to the identity and hence a diffeomorphism. 
Since $\tilde\psi_t$ stays close to $\bar \psi$ in  $C^\infty$-topology, up to choosing a smaller $\epsilon$, \eqref{exp} yields
$$
\|\Delta_{\tilde \psi_t}\tilde \psi_t-2d(({\rm tr}\, \tilde T_t)*_{\tilde \psi_t}\tilde \psi_t))+\mathcal{L}_{V(\tilde\psi_t)}\tilde\psi_t\|_{C^{\infty}_{\tilde g_t}}\leq 2\|\Delta_{\tilde \psi_t}\tilde \psi_t-2d(({\rm tr}\, \tilde T_t)
*_{\tilde \psi_t}\tilde \psi_t))+\mathcal{L}_{V(\tilde\psi_t)}\tilde\psi_t\|_{C^{\infty}_{\bar g}}\leq 2\kappa {\rm e}^{-\lambda t}
$$
where $\tilde g_t$ is the metric induced by $\tilde \psi_t$.
By  diffeomorphism invariance it follows that 
\begin{equation}
\label{stima}
\|\partial_t \psi_t\|_{C^{\infty}_{g_t}}=\|\Delta_{\psi_t}\psi_t-2d(({\rm tr}\, T_t)*_{\psi_t}\psi_t))\|_{C^{\infty}_{g_t}} \leq 2\kappa {\rm e}^{-\lambda t}
\end{equation}
where $g_t$ is the metric induced by $\psi_t$.

Now if we write 
$$
\partial_t\psi_t=\alpha_t\wedge*\psi_t+3*_{\psi_t}\mathsf{i}_{\psi_t}(h_t)
$$ 
we have in particular
$$
\|\partial_t\psi_t\|_{C^{0}_{g_t}}^2=\|\alpha_t\wedge*_{\psi_t}\psi_t\|_{C^{0}_{g_t}}^2+\|3*_{\psi_t}\mathsf{i}_{\psi_t}(h_t)\|_{C^{0}_{g_t}}^2\,.
$$

Thus \eqref{stima} implies $\|h_t\|_{C^{0}_{g_t}}^2\leq  C\kappa {\rm e}^{-\lambda t}\,,$ where $C$ is a positive universal constant.
Moreover by \cite[Proposition 3.1]{Gri}
$$
\partial_tg_t=\frac12 ({\rm  tr}_{g_t}\, h_t)\,g_t-2h_t\,,
$$ 
so that for a vector field $X\neq 0$ on $M$ we have 
$$
|\partial_tg_t(X,X)|\leq \left(\tfrac12|{\rm  tr}_{g_t}\, h_t|+2\|h_t\|_{C^0_{g_t}}\right) \,g_t(X,X)
\leq  C\|h_t\|_{C^0_{g_t}}\,g_t(X,X) 
$$
for some constant $C>0$ and  integrating $\frac{\partial_tg_t(X,X)}{g_t(X,X)}$ we deduce
$$
{\rm e}^{-\frac{C}{\lambda}}g_0\leq g_t\leq {\rm e}^{\frac{C}{\lambda}}g_0
$$
so that the metrics $g_t$ and $g_0$ are uniformly equivalent. It follows that  $g_t$ is uniformly equivalent to $\bar g$ and so 
$\|\partial_t\psi_t\|_{C^0_{\bar g}}\leq C {\rm e}^{-{\lambda t}}$ for some constant $C>0$. 
Thus $\psi_t$ converges in $C^0_{\bar g}$-norm to some $4$-form $\psi_\infty$. 
On the other hand, by \eqref{<>} we have
$$
\begin{aligned}
|\psi_{\infty}-\phi_{\infty}^*\bar \psi|_{\bar g} & \leq\lim_{t\to \infty}\left( |\psi_{\infty}-\psi_t|_{\bar g}+|\psi_{t}-\phi_t^*\bar \psi|_{\bar g}+|\phi_t^*\bar \psi-\phi_{\infty}^*\bar \psi |_{\bar g}\right)\\
 & \leq\lim_{t\to \infty}\left( |\psi_{\infty}-\psi_t|_{\bar g}+C|\tilde\psi_{t}-\bar \psi|_{\bar g}+|(\phi_t^*-\phi_\infty^*)\bar\psi|_{\bar g}\right)=0\,,
\end{aligned}
$$
so that 
$\psi_\infty =  \phi_{\infty}^*\bar \psi$.

The last part consists in showing that $\psi_t$ converges to $\psi_{\infty}$ in $C^{\infty}$-topology. 
We just describe the procedure and refer to \cite{L2} for details. 
First we have exponential estimates for the $\bar g$-norm of the curvature $\tilde R_t$ of $\tilde g_t$ and the first covariant derivatives of the torsion $\tilde T_t$ of $\tilde \psi_t$. Then for $t$ large enough we deduce corresponding estimates with respect to the $g_t$-norms (since $g_t$ is uniformly equivalent to $\bar g$) and finally by diffeomorphism invariance we have uniform bounds for the $g_t$-norm of the curvature $R_t$ of $g_t$ and the covariant derivative of the torsion $T_t$ of $\psi_t$.
This allows us to use the a-priori Shi-type estimates for the Laplacian co-flow \cite[Theorem 2.1]{shi-type} (Note that the flow \eqref{A=0} is a {\em reasonable} flow of ${\rm G}_2$-structures in the sense of \cite{shi-type}).
Now the lower bound on the injectivity radius of $g_t$ (again by uniform equivalence) and the compactness theorem for $ {\rm G}_2$-structures \cite[Theorem 7.1]{L1} gives us the convergence of $\psi_t$ to $\psi_\infty$ in $C^{\infty}$-topology. 
\end{proof}

\subsection{Examples and Remarks}
It is known that in the case $A>0$, the modified Laplacian coflow may have some stationary points which are not torsion-free. Here we observe that 
such stationary points are not stable in general.  A class of examples is provided by {\em nearly parallel } ${\rm G_2}$-structures
which are characterized by the equations
\begin{equation}\label{NP}
d\psi=0 \quad \mbox{and} \quad d(*_\psi\psi)=\tau_0\psi\,,
\end{equation}
where $\tau_0$ is a constant.
For instance the standard ${\rm G}_2$-structure on the $7$-sphere ${\rm Spin}(7)/{\rm G}_2$ is nearly parallel. 
Let us study the evolution of a nearly parallel ${\rm G}_2$-structure $\bar\psi $  by equation \eqref{modified coflow} with $A \geq 0$.
For $\psi$ nearly parallel one has  
$$
\Delta_{\psi}\psi+2d((A-{\rm tr}\, T_{\psi})*_\psi \psi)=
\tau_0
\left(2A-\tfrac{5}{2}\tau_0\right)\psi
$$
It is immediate to note that if the torsion form $\tau_0$ is $\frac45 A$ then $\bar\psi $ is stationary for \eqref{modified coflow} (see also \cite{Lotaysurvey}).
In general the modified Laplacian coflow starting from a nearly parallel ${\rm G}_2$-structure $\psi_0$
acts by rescaling $\psi_t=c_t\psi_0$, where $c_t$ solves the ODE 
\begin{equation}\label{nearly}
\tfrac{d}{dt}c_t=c_t^{3/4}\tau_0\left(2A-\tfrac{5}{2}c_t^{-1/4}\tau_0\right)\,. 
\end{equation}
where $\tau_0$ is the torsion form of $\psi_0$ defined by \eqref{NP}.
Now consider $A > 0$ fixed and take $\bar\psi$ nearly parallel and stationary.
Take $\psi_0 = \mu \bar\psi$ with $\mu$ a real constant.
Now if $\mu > 1$ we have that $c_t$ is increasing, while if $\mu<1$ we have that $c_t$ is decreasing.
In both cases the flow steps away from the stationary solution and $\bar \psi$ is unstable. \\
We can use the same computation to illustrate another phenomenon.  
Grigorian noted in \cite{Gri} that the volume is increasing along the modified Laplacian coflow \eqref{modified coflow} if and only if the following inequality
is satisfied for every $t$
$$
|T_t|^2+{\rm tr}\, T_t (4A-3 {\rm tr}\,T_t)>0\,. 
$$
In the case $A=0$, the volume may decrease. 
Indeed in the situation above we have 
$$
\tfrac{d}{dt}c_t=-\tfrac{5}{2}c_t^{1/2}\tau_0^2\,,
$$
i.e.
$$
c_t=(1-\tfrac{5}{4}\tau_0^2\,t)^2\,.
$$
Since 
$$
{\rm Vol}_{\psi_t}=c_t^{7/4} {\rm Vol}_{\psi_0}
$$
the volume decreases.

\medskip
Next we focus on examples of static solutions to the modified Laplacian coflow which are not torsion free. To construct such examples we consider nilpotent Lie groups and we work in their Lie algebras in an algebraic fashion.
\begin{ex}\label{EE1}{\em 
Let $M=\mathbb T^{4}\times \mathbb{H}^3/\Gamma$, where $\mathbb T^{4}$ is the $4$-dimensional torus, $\mathbb H^3$ is the $3$-dimensional Heisenberg Lie group 
$$
\mathbb{H}^3=\left\{\left[\begin{smallmatrix}1&x&z\\0&1&y\\0&0&1\end{smallmatrix}\right]\mid x,y,z, \in \mathbb R\right\}
$$ 
and $\Gamma$ is the co-compact lattice of matrices in $\mathbb{H}^3 $ with integral entries. Notice that $M$ can be regarded as the product of the Kodaira-Thurston 
manifold with the $3$-dimensional  torus and it is a $2$-step nilmanifold admitting a global coframe  $\{e^1,\dots, e^7\}$
which satisfies 
$$ 
de^i=0, \quad i=1,2,3,4,5,7 \quad \quad de^6=e^1 \wedge e^7. 
$$ 
Let 
$$
\bar\varphi= e^{123}+e^{145}+e^{167}+e^{246}-e^{257}-e^{347}-e^{356}
$$
be the \lq\lq standard\rq\rq ${\rm G}_2$-structure with respect to the co-frame we have fixed. 
(As usual we denote by $e^{ijk\ldots}$ the form 
$e^i\wedge e^j\wedge e^k\wedge \ldots$). An easy computation implies that $\bar\varphi$ is co-closed and that 
$$
\bar \psi=*_{\bar\varphi}\bar\varphi=e^{4567}+e^{2367}+e^{2345}+e^{1357}-e^{1346}-e^{1256}-e^{1247}
$$
is a static solution to \eqref{A=0}. More generally it can be noticed that every left-invariant ${\rm G}_2$-structure $M$ of the form 
$$
\varphi=c_1e^{123}+c_2e^{145}+c_3e^{167}+c_4e^{246}-c_5e^{257}-c_6e^{347}-c_7e^{356}
$$
gives a static solution to \eqref{A=0} and that every left-invariant co-closed ${\rm G}_2$-structure on $M$ is static. 
}
\end{ex}
\begin{ex}\label{EE2}{\em 
Let $\mathfrak g$ be the nilpotent Lie algebra admitting a coframe $\{e^1,\dots,e^7\}$ satisfying 
$$
d e^i=0,\quad  i=1,3,5,7\,,\quad de^2=-e^{13}\,,\quad de^4=e^{15}\,,\quad de^6=e^{17}
$$
and let $G$ be the simply-connected Lie group having $\g$ as Lie algebra. Then $G$ has a co-compact lattice $\Gamma$ and we set  $M=G/\Gamma$. 
A direct computation gives that the standard ${\rm G}_2$-structure 
$$
\varphi=e^{123}+e^{145}+e^{167}+e^{246}-e^{257}-e^{347}-e^{356}
$$
is coclosed and static with respect to the modified Laplacian coflow with $A=0$.  However, in contrast with the previous example, 
we have that on $M$ there are left-invariant coclosed ${\rm G}_2$-structures which are not static. For instance if we consider 
$$
\varphi=c_1^2e^{123}+c_2^2e^{145}+c_3^2e^{167}+c_4^2e^{246}-c_5^2e^{257}-c_6^2e^{347}-c_7^2e^{356}
$$
for $c_i$ constant, then $\varphi$ is coclosed and the corresponding $\psi$ satisfies  
$$
\Delta_{\psi}\psi-2d({\rm tr}\, T_{\psi}*_\psi \psi)=
\frac{2(c_{2}c_{4}c_{7}+c_{2}c_{5}c_{6}-c_{3}c_{4}c_{6})}{c_1^2c_{3}c_{5}c_{7}}\,e^{1357}\,.
$$
}
\end{ex}

\begin{rem}{\em 
Note that the ${\rm G}_2$-structures in Example \ref{EE1} and Example \ref{EE2} cannot be torsion-free since a (compact) nilmanifold $M$ cannot admit a 
left-invariant ${\rm G}_2$-structure unless it is a torus. 
}
\end{rem}

\section{From theorem \ref{teoremone} to the stability of the Balanced flow}\label{3}
We recall that given a K\"ahler manifold $(M,\omega_0)$ the Calabi-flow starting from  $\omega_0$ is the geometric flow of K\"ahler forms governed by the equation 
\begin{equation}\label{CF}
\partial_t\omega_t=i\partial \bar\partial s_{\omega_t}\,,\quad \omega_{|t=0}=\omega_0
\end{equation}
where $s_{\omega_t}$ is the Riemannian scalar curvature of $\omega_t$.  Many properties of the flow were proved in \cite{chen}. Here we recall the theorem by Chen and He about the stability of the Calabi-flow. 

\begin{theorem}[Chen-He]\label{Chen-He}
Let $(M,\bar \omega)$ be a  compact K\"alher manifold with constant scalar curvature. Then there exists $\delta>0$ such that if $\omega_0$ is a K\"ahler metric satisfying 
$$
\|\omega_0-\bar\omega\|_{C^{\infty}}<\delta\,,
$$  
then the Calabi-flow starting from $\omega_0$ is immortal and converges in $C^{\infty}$ topology to a constant scalar curvature K\"ahler metric in $[\bar \omega].$  
\end{theorem}

The Calabi-flow was generalized to the context of balanced geometry  by the authors in \cite{Wlucione} (see also \cite{Wlucione2} for a generalizations
 in a different direction). A Hermitian metric on a complex manifold is called {\em balanced} if its fundamental form is co-closed (instead  of closed as in the K\"ahler case).  Given a compact balanced manifold $(M,\omega_0)$ of complex dimension $n$ the {\em balanced flow} consists in evolving $\omega_0$ as  
\begin{equation}\label{bflow}
\begin{cases}
& \partial_t*_t\omega_t=i\partial\bar{\partial} *_{t}(\rho_t \wedge \omega_t)+(n-1)\Delta_{BC}^t*_{t}\omega_t\\
&d\omega_t^{n-1}=0\\
& \omega_{|t=0}=\omega_0\,,
\end{cases}
\end{equation}
where $*_t$, $\rho_t$  and $\Delta_{BC}^t$  are the Hodge star operator, the Chern-Ricci form and the modified Bott-Chern Laplacian of $\omega_t$, respectively (see \cite{Wlucione} for details). Also this flow fits in the class of flows of theorem \ref{teoremone} 
when we consider the following Hodge system:
$$
\begin{CD}
\Omega^{n-2,n-2} @>D=i\partial \bar\partial >>\Omega^{n-1,n-1} \\
@VV\Delta_D=\Delta_A V\\
\Omega^{n-2,n-2} @<D^*<<\Omega^{n-1,n-1}
\end{CD}
$$
where $\Delta_A$ is the modified Aeppli Laplacian
$$\Delta_{A}:=\ov{\partial}^*\partial^*\partial\ov{\partial}+\partial\ov{\partial}\ov{\partial}^*\partial^*+\partial\ov{\partial}^*\ov{\partial}\partial^*+\ov{\partial}\partial^*\partial\ov{\partial}^*+\partial\partial^*+\ov{\partial}\ov{\partial}^*\,.$$
\begin{theorem}\label{stabbalanced}
Let $(M,\bar \omega)$ be a compact Ricci-flat K\"ahler manifold of complex dimension $n$. Then there exists $\delta>0$ such that if $\omega_0$ is a balanced metric on $M$ satisfying $\|\omega_0-\bar \omega\|_{C^{\infty}}<\delta$, then flow \eqref{bflow} with initial datum $\omega_{|t=0}=\omega_0$ exists for all $t\in [0,\infty)$ and as $t\to \infty$ it 
converges in $C^{\infty}$ topology to a balanced form $\omega_{\infty}$ satisfying 
$$
i\partial\bar{\partial} *_{\omega_{\infty}}(\rho_{\omega_{\infty}} \wedge \omega_{\infty})+(n-1)\Delta_{BC}^{\omega_\infty}*_{\omega_{\infty}}\omega_{\infty}=0\,. 
$$
\end{theorem}

\begin{proof}
A Hermitian form $\omega$ on a complex manifold $M$ is determined by an $(n-1,n-1)$-form $\varphi$ which is positive in the sense that 
\begin{equation}\label{balancedform}
\varphi(Z_1,\dots, Z_{n-1},\bar Z_1,\dots, \bar Z_{n-1})>0
\end{equation}
for every $\{Z_1,\dots,Z_{n-1}\}$ linearly independent vector fields of type $(1,0)$ on $M$ (here $n$ is the complex dimension of $M$). Indeed, once such a form $\varphi$ is given, there exists a unique Hermitian form $\omega$ such that $*_{\omega}\omega=\varphi$.  We denote by $\mathcal E\subseteq \Lambda^{n-1,n-1}_\R$ the bundle whose sections are real $(n\!-\!1,n\!-\!1)$--forms satisfying \eqref{balancedform}. Flow \eqref{stabbalanced} can be alternatively written in terms of $\varphi$ as 
\begin{equation}\label{theflow}
\begin{cases}
& \partial_t\varphi_t=i\partial\bar{\partial} *_t(\rho_t \wedge *_t\varphi_t)+(n-1)\Delta_{BC}^t\varphi_t\\
&d\varphi_t=0\\
& \varphi_{|t=0}=\varphi_0\,,
\end{cases}
\end{equation}
and we denote by $Q\colon C^{\infty}(M,\mathcal E)\to \Omega^{n-1,n-1}_\R$ the operator 
$$
Q(\varphi)=i\partial\bar{\partial} *_{\varphi}(\rho_{\varphi} \wedge *_\varphi\varphi)+(n-1)\Delta_{BC}^\f \varphi\,. 
$$
In order to apply theorem \ref{teoremone} we show that for a Ricci-flat K\"ahler form $\bar\omega$ on $M$ the corresponding $\bar\f=*_{\bar\omega} \bar \omega$ is such that
\begin{enumerate}
\item[1.] $Q(\bar \f)=0$;
\vspace{0.2cm}
\item[2.] the restriction to $D\,\Omega^{n-2,n-2}$ of $L_{\bar \varphi}$ is symmetric and negative definite with respect to the $L^2$ inner product induced by $\bar \omega$\,.
\end{enumerate}
Item 1 is trivial and item 2 can be deduced from \cite[Section 5]{Wlucione}, but we prove it for the sake of completeness.

Let $\{\omega_t\}_{t\in (-\epsilon,\epsilon)}$, be a smooth curve of  balanced forms which is $\bar \omega$ at $t=0$ and such that the corresponding $(n\!-\!1,n\!-\!1)$-forms $\varphi_t$ are in the Bott-Chern cohomology class of $\bar\varphi$, and let 
$$
\chi=\partial_{t|t=0}*_{t}\omega_t\,.
$$
Then we can write 
$$
\chi=h_1\bar\f+*_{\bar \omega} h_0
$$
for a smooth function $h_1$ and a $(1,1)$-form $h_0$ such that $h_0\wedge\omega^{n-1}=0$. In this way 
$$
\partial_{t|t=0}\omega_t=\frac{h_1}{n-1}\bar \omega-h_0\,
$$
see \cite[Lemma 2.5]{Wlucione}. Since $\bar \rho=0$ we have  
$$
\partial_{t|t=0} i\partial\bar{\partial} *_{t}(\rho_t \wedge \omega_t)= i\partial\bar{\partial} *_{\bar \omega}(\dot \rho \wedge \bar \omega)
$$
where we have set 
$
\dot \rho= \partial_{t|t=0} \rho_{t}\,. 
$
In view of \cite[lemma 5.1]{Wlucione} $\dot\rho=-in\partial	\bar\partial h_1$ and so 
$$
\partial_{t|t=0} i\partial\bar{\partial} *_t(\rho_t \wedge \omega_t)=n\partial\bar{\partial} *_{\bar \omega}(\partial\bar\partial h_1\wedge \omega)\,. 
$$
On the other hand it is clear that  for a curve of Bott-Chern-cohomologous $(n\!-\!1,n\!-\!1)$-forms $\varphi_t$ starting at $\bar \varphi$ we have 
$$
\partial_{t|t=0}\Delta^{\varphi_t}_{BC}\varphi_t=- \,\partial \bar\partial *_{\bar \omega} \partial \bar\partial\, \dot \omega=- \,\partial \bar\partial *_{\bar \omega} \partial \bar\partial\,\left(\frac{h_1}{n-1}\bar \omega-h_0\right)\,.
$$ 
And then we obtain 
\begin{eqnarray*}
L_{\bar\varphi}(\psi) & = &\partial_{t|t=0}i\partial\bar{\partial} *_t(\rho_t \wedge \omega_t)+(n-1)\partial_{t|t=0}\Delta^{\varphi_t}_{BC}\varphi_t=(n-1)\partial \bar\partial *_{\bar \omega} \partial \bar\partial\,\left(h_1\bar \omega-h_0\right)
\end{eqnarray*}
Moreover since $\bar\omega$ is K\"ahler we have 
$$
\begin{aligned}
\Delta^{\bar\omega}_{BC}(\psi) & =  - \dd*_{\bar\omega}\dd*_{\bar\omega} \psi  =  - \dd*_{\bar\omega}\dd*_{\bar\omega}(h_1\bar\f+*_{\bar \omega} h_0) \\
& =  - \dd*_{\bar\omega}\dd(h_1\bar\omega) + \dd*_{\bar\omega}\dd h_0 
 =  - \dd*_{\bar\omega}\dd( h_1 \w \bar\omega -h_0)
\end{aligned}
$$
and so 
$$
L_{\bar\varphi}(\psi) =-(n-1)\Delta^{\bar\omega}_{BC}(\psi)
$$
which implies the statement. 
\end{proof}
 
\begin{rem} 

\medskip 
{\em It is quite natural wondering if theorem \ref{stabbalanced} can be improved by showing that the limit balanced metric 
is actually Calabi-Yau, or by proving the stability around more general static solutions of the flow, such as constant scalar curvature K\"ahler metrics or even balanced metrics $\omega$ satisfying
\begin{equation}\label{GB}
i\partial\bar{\partial} *(\rho \wedge  \omega)+(n-1)\Delta_{BC}*\omega=0\,. 
\end{equation}
These improvements cannot be easily deduced from our theorem \ref{teoremone} and will be the subject of some future studies. 
On the other hand, theorem \ref{teoremone} suggests to consider balanced metrics satisfying \eqref{GB} as natural generalizations of extremal K\"ahler metrics to 
the  context of balanced geometry (for a generalization in another direction see \cite{fei}) and the problem of the existence and uniqueness of such metrics in a fixed Bott-Chern cohomology class arises. More general it could be interesting to compare the geometry of extremal K\"ahler metrics to the geometry of balanced metrics satisfying \eqref{GB}.
}  
\end{rem}

\section{Proof of theorem \ref{teoremone}} 
In this last section we prove theorem \ref{teoremone}. 
The scheme of the proof resembles the one of the main theorem of \cite{witt1} and of \cite{Wlucione2}.

\smallskip
First we need to recall some basic facts about the category of tame Fr\'echet spaces and tame maps (see \cite{HamiltonNash} for the relevant details).
A {\em tame Fr\'echet space} is a vector space $\mathcal V$ endowed with  a topology given by an increasing countable family of seminorms $\{|\cdot|_n\}$. Thus a sequence $\{x_n\}\subseteq \mathcal V$ will be  {\em convergent} if it converges with respect to each seminorm. A continuous map $F\colon (\mathcal V,|\cdot|_n)\to (\mathcal W,|\cdot|'_n)$ between two tame Fr\'echet spaces is called {\em tame} if for every $x\in \mathcal V$  there are a neighborhood $U_x$ of $x$, a natural number $r$ and positive numbers $b,C_n$ such that
$$
|F(y)|'_n\leq C_n(1+|y|_{n+r})
$$
for every $y\in U_x$ and $n>b$. A differentiable map between tame Fr\'echet spaces is called {\em smooth tame} if all its derivatives are tame maps.
The main relevant result is the celebrated {\em Nash-Moser  theorem.}

\begin{theorem}[\bf{Nash-Moser}]\label{Nash-Moser}
Let $\mathcal{V}$, $\mathcal{W}$ be tame Fr\'echet spaces and let  $\mathcal{U}$
be an open subset of  $\mathcal{V}$. 
Let $F\colon \mathcal U\to \mathcal W$ be a smooth map. If the differential  of $F$, $F_{*|x}\colon \mathcal V\to \mathcal W$, is an isomorphism  for every $x\in\mathcal{U}$ and the map $(x,y) \mapsto F_{*|x}^{-1}y$ is smooth tame, then $F$ is locally invertible with smooth tame local inverses.
\end{theorem}

Let $\pi\colon E\to M$ be a vector bundle over a compact oriented Riemannian manifold $(M,g)$ with a metric $h$ along its fibres. Once a connection $\nabla$ on $E$ preserving $h$ is fixed, the space $C^{\infty}(M,E)$ of global smooth sections of $E$ has a natural structure of tame Fr\'echet space given by the Sobolev norms $\|.\|_{H^n}$ induced by $h$, $\nabla$ and the volume form of $g$.
Fix now a closed interval $[a,b]$ and consider the space of time-dependent partial differential operators $P\colon C^{\infty}(M \times [a,b],E)\to C^{\infty}(M \times [a,b],E)$
having degree at most $r$.  This space is tame Fr\'echet with respect to the family of seminorms
$$
|[P]|_n=\sum_{jr\leq n} [\partial_t^{j}P]_{n-jr}
$$
where $[P]_n$ is the supremum  of the norm of $P$ and its space covariant derivatives 
up to degree $n$.

\medskip 
 Now we can focus on the setting described in the introduction considering a Hodge system $(E_-,E,D,\Delta_D)$ on $M$, $\mathcal E, D_+,Q$ as in section \ref{2} and studying flow \eqref{problem} under the assumptions 1, 2, 3.  Let us fix a connection $\nabla$ on $E$ and define the spaces  
$$
\mathcal{F}[a,b]=DC^{\infty}(M\times [a,b],E_{-})\,,\quad \mathcal{G}[a,b]=\mathcal{F}[a,b]\times DC^{\infty}( M,E_-)\,.
$$
Both the above spaces have a structure of tame Fr\'echet spaces given by the gradings 
$$
\|\beta\|_{\mathcal{F}^n[a,b]}=\sum_{2rj\leq n}\int_{a}^{b}\|\partial_t^{j}\beta_t \|_{H^{n-2rj}}\, dt
$$
and 
$$
\|(\beta,\sigma)\|_{\mathcal{G}^n[a,b]}=\|\beta \|_{\mathcal{F}^n[a,b]}+\|\sigma\|_{H^{n}}
$$
respectively. 
If we fix  $\bar\eta \in \Phi=\ker D_+\cap C^{\infty}(M,\mathcal E)$, then we can set
$$
\mathcal{U}=\{\beta\in \mathcal{F}[a,b]\,\,:\,\, \bar\eta+\beta_t \in U \mbox{ for every }t\in [a,b]\}\,,
$$
where, accordingly to section \ref{2}, $U=\{\bar\eta+D\gamma \,\,:\,\,\gamma\in C^\infty(M,E_-)\}\,\cap C^{\infty}(M,\mathcal E)$. Note that 
$\mathcal U$ is open in $\mathcal{F}[a,b]$.
 
 \medskip 
The following theorem is proved in \cite{positive} for second order operators and then extended in \cite{Wlucione} to operators of arbitrary degree.  

\begin{theorem}\label{CRELLE}
Let
$$
F\colon \mathcal{U}\to\mathcal G[a,b]\,,\quad F(\beta)=\left(\partial_t \beta-Q(\bar\eta+\beta),\beta_a \right)\,.
$$
Then 
\begin{enumerate}
\item[1.] $F$ is smooth tame;
\item[2.] $F_{*|\beta}$ is an isomorphism for every $\beta\in\mathcal{U}$;
\item[3.] the map $\mathcal{U}\times \mathcal{G}[a,b]\to \mathcal{F}$, $(\beta,\psi)\mapsto F_{*|\beta}^{-1}\psi$, 
is smooth tame.
\end{enumerate}
\end{theorem}

The starting point of the proof of theorem \ref{teoremone} is the following weak stability result that is a consequence of theorem \ref{CRELLE}.

\begin{prop}
\label{weakstability}
Let $\bar\varphi\in Q^{-1}(0)$. For every $T>0$  and $\varepsilon>0$,  there exists $\delta>0$ such that if $\varphi_0\in \Phi$ and satisfies  
$$
\varphi_0-\bar \varphi\in DC^{\infty}(M,E_{-})\,,\quad \|\varphi_0-\bar\varphi\|_{C^\infty}\leq \delta\,, 
$$
then there exists a smooth solution $\{\varphi_t\}_{t \in [0,T]}$ to \eqref{problem} such that 
$$
 \|\varphi-\bar \varphi\|_{\mathcal F^{n}[0,T]}\leq \varepsilon\,,\mbox{ for every }n\in \mathbb N\,.  
$$
\end{prop}
\begin{proof}
We use theorem \ref{CRELLE} with $\bar\eta=\bar\varphi$. 
Since $F(0)=(0,0)$, theorem \ref{CRELLE} together with Nash-Moser theorem \ref{Nash-Moser} implies that there exist an open neighborhood $\mathcal U'$ of $0$ in $\mathcal U$ and an open 
neighborhood $\mathcal V'$ of $(0,0)$ in $\mathcal G$ such that $F\colon \mathcal U'\to \mathcal V'$ is invertible with smooth tame inverse. By choosing $\delta$ small enough we may assume that $(0,\varphi_0-\bar\varphi) \in \mathcal V'$. So we can take $\beta_t \in \mathcal U'$ such that $F(\beta_t)=(0,\varphi_0-\bar\varphi)$. Hence $\varphi_t=\bar\varphi+\beta_t$ satisfies 
$$
\partial_t\varphi_t=Q(\varphi_t)\,,\quad \varphi_{|t=0}=\varphi_0\,. 
$$
Since $F^{-1}$ is continuous, if we fix $\varepsilon>0$ and we choose $\delta$ small enough, we have $\|\varphi-\bar\varphi\|_{\mathcal F^n[0,T]}\leq \varepsilon$ for every $n\in \mathbb N$ and the claim follows. 
\end{proof}


The next step is the following
\begin{lemma}[Interior Estimate]
\label{intest_gen}
For every $n,T>0$ and $\epsilon \in (0,T)$, there exists $\delta , C> 0$ and and $l=l(n) \in \mathbb{N}$, with $C$ depending on $T,\epsilon$ and an upper bound on $\delta$ such that if $\{\varphi_t\}_{t\in[0,T ]}$ is a smooth curve in $U$ with 
$$
\| \varphi-\bar\varphi\|_{\mathcal F^{l}[0,T]} \leq \delta
$$
and $\sigma\in \mathcal F[0,T]$ satisfies 
$$
\partial_t\sigma=L_{\varphi}\sigma\,,
$$
then 
\begin{equation}\label{induction}
\|\sigma\|_{\mathcal{F}^{2nr}[t_0+\epsilon ,T]}\leq C\|\sigma\|_{\mathcal{F}^0[t_0 ,T]}\,,
\end{equation}
for every $t_0 \in [0,T-\epsilon]$.
\end{lemma}

\begin{proof}
We prove the statement by induction on $n$. For $n=0$ the claim is trivial and we assume the statement true up to $N$. 
From \cite[lemma 4.6]{Wlucione} there exist $\delta, C \in \R^+$, with
$C$ depending on $T$ and an upper bound on $\delta$, such that for $t\in [0,T)$ and $\sigma\in \mathcal F[0,T]$ we have
\begin{eqnarray*}
\|\sigma\|_{ \mathcal{F}^{N+2r}[t,T]} & \leq & C\left( \|\partial_t\sigma-L_{\varphi}(\sigma)\|_{ \mathcal{F}^{N}[0,T]}
+\|\sigma_{0}\|_{H^{N+r}}\right)\\
 & & +C|[L_{\varphi}]|_N \left( \|\partial_t\sigma-L_{\varphi}(\sigma)\|_{ \mathcal{F}^{0}[0,T]}+\|\sigma_{0}\|_{H^{N}}\right) \\
 & \leq & (1+C)|[L_{\varphi}]|_N \left( \|\partial_t\sigma-L_{\varphi}(\sigma)\|_{ \mathcal{F}^{N}[0,T]}
+\|\sigma_{0}\|_{H^{N+r}}\right)\,.
\end{eqnarray*}
If $\| \varphi-\bar\varphi\|_{\mathcal F^{l}[0,T]}< \delta$ for $l$ big enough then $|[L_{\varphi}]|_N \leq 1+|[L_{\bar\varphi}]|_N$, so we have 
\begin{equation}\label{est1}
\|\sigma\|_{ \mathcal{F}^{N+2r}[t,T]}\leq (1+C)(1+ |[L_{\bar\varphi}]|_N)\left( \|\partial_t\sigma-L_{\varphi}(\sigma)\|_{ \mathcal{F}^{N}[0,T]}+\|\sigma_{0}\|_{H^{N+r}}\right)
\end{equation}
for every $\sigma \in \mathcal F[0,T]$. 
Now take $\sigma\in \mathcal F[0,T]$ solution of the linear equation $\partial_t\sigma=L_{\varphi}\sigma$  and fix $\epsilon \in (0,T)$. Choose a smooth function $\chi\colon \R \to [0,1]$ such that 
$$
\chi(t)=0 \quad \mbox{ for } t \leq t_0+\epsilon/2\,,\quad \chi(t)=1\quad \mbox{ for } t \geq  t_0+\epsilon\,.
$$
Set $\tilde \sigma=\chi \sigma$. Then $\partial_t\tilde \sigma =\dot \chi \sigma+\chi \partial_t \sigma$ and 
$$
\partial_t \tilde \sigma-L_\varphi(\tilde \sigma)=\dot \chi \sigma\,.
$$
Hence  using \eqref{est1} we find $C'>0$ depending only on $\epsilon$ such that
\begin{multline*}	
$$
\|\sigma\|_{\mathcal F^{2r(N+1)}[t_0+\epsilon,T]}\leq\|\tilde \sigma\|_{\mathcal F^{2r(N+1)}[t_0+\epsilon/2,T]}\leq  (1+C)(1+ |[L_{\bar\varphi}]|_N) \|\dot \chi \sigma\|_{\mathcal F^{2rN}[t_0+\epsilon/2,T]} \\
\leq C'(1+C)(1+ |[L_{\bar\varphi}]|_N)\,\| \sigma\|_{\mathcal F^{2rN}[t_0+\epsilon/2,T]}
$$
\end{multline*}
and the induction assumption implies the statement. 
\end{proof}



Now we need a general lemma for families of symmetric operators on Hilbert spaces.
Here we will say that, given a Hilbert space $\mathcal{H}_1$ continuously embedded in a Hilbert space $\mathcal{H}_2$, an operator $L: \mathcal{H}_1 \to \mathcal{H}_2$ is symmetric if $\langle L z_1,z_2\rangle_{\mathcal{H}_2} = \langle z_1,L z_2 \rangle_{\mathcal{H}_2}$ for every $z_1, z_2 \in \mathcal{H}_1\,$.
Analogously we will say that $L$ is negative semidefinite if  $\langle L z,z\rangle_{\mathcal{H}_2}\leq 0$ for every $z \in {\mathcal{H}_1}$.

\begin{lemma}
\label{Hilbert}
Let $(X,\bar x)$ be a pointed metric space and let $\mathcal{H}_1$ and $\mathcal{H}_2$ be two Hilbert spaces with $\mathcal{H}_1$ continuously embedded in $\mathcal{H}_2$.
Let $\{L_x\}_{x \in X}$ be a continuous family of bounded symmetric operators $L_x :  \mathcal{H}_1 \to \mathcal{H}_2$. 
Assume that $L_{\bar x}$ is negative semidefinite and that there exists $C>0$ such that 
\begin{equation}
\label{boh}
\|z_0\|_{\mathcal{H}_1} \leq C \| z_0\|_{\mathcal{H}_2}\,, \quad \mbox{for every $z_0 \in \ker L_{\bar x}\,,$}
\end{equation}
and
\begin{equation}
\label{inversa}
\|z_1\|_{\mathcal{H}_1} \leq C \|L_{\bar x} z_1\|_{\mathcal{H}_2}\,, \quad \mbox{for every $z_1 \in (\ker L_{\bar x})^{\perp}\,.$}
\end{equation}
Then for every $\epsilon>0$ there exists $\delta>0$ such that if $x \in X$ satisfies $d(x,\bar x)<\delta$,
\begin{equation}
\langle L_{x} z ,z\rangle_{\mathcal{H}_2} \leq (1-\epsilon) \langle L_{\bar x} z ,z\rangle_{\mathcal{H}_2}+\epsilon \|z\|^2_{\mathcal{H}_2}
\quad \mbox{for every $z\in \mathcal{H}_1$.}
\end{equation}
\end{lemma}

\begin{proof} 
Fix $\epsilon >0$. Let $T:=-\epsilon L_{\bar x}$ and $V_x:=L_{\bar x}-L_x$, for every $x \in X$. 
Now $T$ is symmetric and positive semidefinite. 
Let us write $z=z_0+z_1$ according to the decomposition $\mathcal{H}_1=\ker L_{\bar x} \oplus (\ker L_{\bar x})^\perp$. 
Thus for $b>0$ arbitrarily small, using also \eqref{inversa} we can find $\delta>0$ such that if $d(x,\bar x)\leq \delta $, we have
$$
\|V_x z_1\|_{\mathcal{H}_2}\leq b \epsilon C^{-1} \|z_1\|_{\mathcal{H}_1} \leq b \|T z\|_{\mathcal{H}_2}
$$
for every $z\in \mathcal{H}_1$. Consequently using \eqref{boh}, up to shrinking $\delta$ we have
$$
\|V_x z\|_{\mathcal{H}_2}\leq  \|L_x z_0 \|_{\mathcal{H}_2} + \|V_x z_1\|_{\mathcal{H}_2} \leq a\|z\|_{\mathcal{H}_2} + b \|T z \|_{\mathcal{H}_2}
$$
with $a>0$ arbitrarily small. Taking $a=\frac{\epsilon}{2}$ and $b=\frac12$ and using \cite[Theorem 9.1]{weidmann} we have 
that $\langle (T +V_x) z, z \rangle_{\mathcal{H}_2} \geq -\epsilon \|z\|_{\mathcal{H}_2}^{2}$ for every $z \in {\mathcal{H}_1}$ and the claim follows.
\end{proof} 

Now we apply the previous lemma to the family of operators $L_\varphi$ in the following situation: $\mathcal{H}_1 = H^{2r}(M,E)$, {\em i.e.} the space of sections of $E$ whose local components have square integrable derivatives up to order $2r$, $\mathcal{H}_2 = L^{2}(M,E)$, $(X,\bar x) = (\Phi,\bar \varphi)$ where on $\Phi$ we consider the distance induced by $H^{\bar n}(M,E)$, where $\bar n=\frac{\dim M}{2}+2r+1$. The choice of $\bar n$ ensures via the Sobolev embedding theorem that  $\{L_{\varphi}\}_{\varphi\in \Phi}$ is a {\em continuous} family of bounded operators.  Since inequality \eqref{inversa} comes from Fredholm alternative and inequality \eqref{boh} holds due to elliptic regularity of $L_{\bar \varphi}$ we get the following corollary.


\begin{cor}
\label{corollario_a}
For every $a>0$ there exists $\delta>0$ such that if $\varphi\in C^\infty(M,\mathcal E)$ satisfies $\|\varphi-\bar\varphi\|_{H^{\bar n}}<\delta$, then 
\begin{equation}
\langle L_{\varphi}(z),z\rangle_{L^2} \leq (1-a) \langle L_{\bar\varphi}(z),z\rangle_{L^2}+a \|z\|^2_{L^2}
\end{equation}
for every $z\in H^{2r}(M,E)$. 
\end{cor}


Next we deduce the following trace-type theorem in $C^\infty(M \times [0,T],E_-)$:

\begin{prop}\label{mante} 
For every $n\in \mathbb N$ and $\ell \in \R_+$ there exists positive constants $C$ and $m \in {\mathbb N}$ such that
$$
\|\beta_t\|_{H^n}\leq C\|\beta\|_{\mathcal F^{m}I}
$$ 
for every $\beta\in C^\infty(M \times I,E_-)$, and $t \in I$, where $I\subseteq \R$ is a closed interval of length $\ell$. 
\end{prop}
\begin{proof}
Arguing exactly as in \cite[proposition 4.1]{MM} for every $s \in \mathbb N$ we get the following inequality
$$
\|\nabla^s\beta\|_{C^0(M\times I,E_-)} \leq C\|\beta\|_{\mathcal F^{m}I}
$$
for $m>{\rm max}\{\frac{s+\dim M+2r}{4r},\frac{s}{2r}\}$ and $C$ independent of $\beta$. (Here $\N^s$ denotes spatial derivatives only).
\end{proof}

\begin{cor}\label{cormante}
For every $T>0$, $n\in \mathbb N$ and $\epsilon'>0$ there exist $C>0$ and $m \in {\mathbb N}$ such that every  $\beta \in C^\infty(M \times [0,T+\epsilon'],E_-)$ satisfies 
$$
\|\beta_t\|_{H^n}\leq C\|\beta\|_{\mathcal F^{m}[t,T+\epsilon']}
$$
for every $t\in [0,T]$.
\end{cor}
\begin{proof}
It is enough to apply proposition \ref{mante} with $I=[t,t+\epsilon']$. In this way 
$$
\|\beta_t\|_{H^n}\leq  C\|\beta\|_{\mathcal F^{m}[t,t+\epsilon']} \leq   C\|\beta\|_{\mathcal F^{m}[t,T+\epsilon']}
$$ 
with $C$ independent of $\beta$ and $t$ and the claim follows. 
\end{proof}


\begin{lemma}[exponential decay]\label{Wexp}
Let $\epsilon>0$ and $T>\epsilon $. There exists $\delta>0$ such that 
if $\varphi_0\in \Phi$ satisfies  
\begin{equation}\label{phi0}
\varphi_0-\bar \varphi\in DC^{\infty}(M,E_{-})\,,\quad \|\varphi_0-\bar\varphi\|_{C^\infty}\leq \delta\,, 
\end{equation}
then the solution $\varphi_t$ to \eqref{problem} is defined in  $M\times  [0,T]$ and satisfies
$$
\|Q(\varphi_t)\|_{H^n}\leq C \|Q(\varphi_0)\|_{L^2} \,{\rm e}^{-\lambda t} \,,\mbox{ for every }t\in [\epsilon,T]\,,
$$
where $\lambda$ is half the first positive eigenvalue of $-L_{\bar\varphi}$ and $C$ is a constant depending on $n$, $\epsilon$, $T$ and an upper bound on $\delta$. 
\end{lemma}

\begin{proof}
Fix a small time $\epsilon'>0$ arbitrary. Proposition \ref{weakstability} implies that there exists $\delta>0$ such that if $\varphi_0$ satisfies \eqref{phi0}, then problem \eqref{problem} has a solution $\f\in C^{\infty}(M\times [0,T+2\epsilon'],E)$  with $\|\varphi-\bar\varphi\|_{\mathcal F^l[0,T+2\epsilon']}$ bounded for every $l$. 
Now  
$$
\partial_t^2\varphi_t=\partial_tQ(\varphi_t)\,,
$$
implies
$$
\partial_tQ(\varphi_t)=L_{\varphi_t}Q(\varphi_t)
$$ 
and 
$$
\partial_t\|Q(\varphi_t)\|^2_{L^2}=2 \langle \partial_tQ(\varphi_t),Q(\varphi_t)\rangle_{L^2}=2 \langle L_{\varphi_t}Q(\varphi_t),Q(\varphi_t)\rangle_{L^2}
$$
for every $t\in [0 ,T+2\epsilon']$.  
In view of corollary \ref{cormante} and corollary \ref{corollario_a} we can choose the initial $\delta$ so small that 
for every $t\in [0,T+\epsilon']$ we have
$$
\langle L_{\varphi_t}Q(\varphi_t),Q(\varphi_t)\rangle_{L^2}\leq (1-a) \langle L_{\bar \varphi}Q(\varphi_t),Q(\varphi_t)\rangle_{L^2}+ a \|Q(\varphi_t)\|^2_{L^2}
$$
with $a=\frac{\lambda}{2\lambda+1}\,.$ 
Taking into account that $\langle L_{\bar \varphi} Q(\varphi_t),Q(\varphi_t) \rangle_{L^2} \leq -2\lambda \|Q(\varphi_t)\|^2_{L^2}\,,$ we have
$$
\langle L_{\varphi_t}Q(\varphi_t),Q(\varphi_t)\rangle_{L^2}\leq -{\lambda} \langle L_{\bar \varphi}Q(\varphi_t),Q(\varphi_t)\rangle_{L^2}\,.
$$
So 
$$
\partial_t\|Q(\varphi_t)\|^2_{L^2} \leq -2\lambda \|Q(\varphi_t)\|^2_{L^2}
$$
and by Gronwall's lemma we get 
$$
\|Q(\varphi_t)\|^2_{L^2}\leq {\rm e}^{- 2\lambda t} \|Q(\varphi_0)\|^2_{L^2}
$$
for every $t \in [0,T+\epsilon']$.
We have 
\begin{equation}
\label{exp2}
\|Q(\varphi)\|^2_{\mathcal F^{0}[t,T+\epsilon']}=\int_{t}^{T+\epsilon'}\|Q(\varphi_s)\|_{L^2}^2\,ds\leq \|Q(\varphi_0)\|^2_{L^2}\int_{t}^{T+\epsilon'}{\rm e}^{-2\lambda s} ds\leq \|Q(\varphi_0)\|^2_{L^2}
\frac{{\rm e}^{-2\lambda t}}{2\lambda}\,.
\end{equation}
By corollary \ref{cormante} we find $m$ such that for every $t\in [0,T]$
$$
\|Q(\varphi_t)\|_{H^{n}} \leq C \|Q(\varphi)\|_{\mathcal{F}^{m}[t,T+\epsilon']}\,.
$$
Now by lemma \ref{intest_gen} we can take $l$ big enough such that if $\|\varphi-\bar\varphi\|_{\mathcal{F}^l[0,T+2\epsilon']} \leq\delta$ we have
$$
\|Q(\varphi)\|_{\mathcal{F}^{m}[t,T+\epsilon']} \leq C\|Q(\varphi)\|_{\mathcal{F}^{0}[t-\epsilon,T+\epsilon']}\,,
$$
for every $t\in [\epsilon,T+\epsilon']\,.$
Finally putting these together with \eqref{exp2} we have
$$
\|Q(\varphi_t)\|_{H^{n}} \leq C  \|Q(\varphi_0)\|_{L^2}{\rm e}^{{-\lambda} t}\,
$$
for $t \in [\epsilon,T]$ as required.
\end{proof}

Now we are ready to prove the main theorem.

\begin{proof}[Proof of theorem $\ref{teoremone}$] 
Let $T>0$ and $\epsilon \in (0,\frac{T}{2})$ be fixed. Using theorem \ref{Wexp}, there exists $\delta'>0$ such that  if $\|\varphi_0-\bar\varphi\|_{C^\infty} \leq \delta'$, then the solution $\varphi_t$ to the geometric flow \eqref{problem} exists in $[0,T]$ and for every $n\in \mathbb N$
\begin{equation}
\|Q(\varphi_t)\|_{H^{n}}\leq C\|Q(\varphi_0)\|_{L^2}{\rm e}^{-\lambda t} \mbox{ for every }t\in [\epsilon,T]\,,
\end{equation}
for some $C>0$ depending on $n$, $\epsilon$, $T$ and an upper bound on $\delta'$.

Now we choose $\delta \leq \delta'$ such that  if $\|\varphi_0-\bar\varphi\|_{C^{\infty}}\leq \delta$ then
\begin{equation}\label{condizione}
C\|Q(\varphi_0)\|_{L^2}\frac{{\rm e}^{-\lambda \epsilon}}{\lambda}\sum_{j=0}^{\infty}{\rm e}^{-\lambda j(T-\epsilon)}+\|\varphi_{\epsilon}-\bar\varphi\|_{H^{n}}\leq \delta'\,. 
\end{equation}
We show that $\varphi$ can be extended to $M\times [0,\infty)$ and converges to an element of $U$ lying in the 0 level set of $Q$ as $t\to \infty$.  
We have 
$$
\begin{aligned}
\|\varphi_{t}-\bar\varphi\|_{H^{n}}=&\left\|\int_{\epsilon}^{t}Q(\varphi_\tau)\,d\tau+\varphi_\epsilon-\bar\varphi\right\|_{H^{n}}
\leq \int_{\epsilon}^{t}\|Q(\varphi_\tau)\|_{H^{n}}\,d\tau+\|\varphi_\epsilon-\bar\varphi\|_{H^{n}}\\
& \leq C\|Q(\varphi_0)\|_{L^2}\frac{{\rm e}^{-\lambda\epsilon}}{\lambda}+\|\varphi_\epsilon-\bar\varphi\|_{H^{n}}\,,\quad t \in [\epsilon,T]
\end{aligned}
$$
and condition \eqref{condizione} implies $\|\varphi_{T-\epsilon}-\bar\varphi\|_{H^{n}}\leq \delta'$ and therefore $\varphi$ can be extended  in $M\times [0,2T-\epsilon]$. Moreover,  
$$
\|Q(\varphi_t)\|_{H^{n}}\leq C \|Q(\varphi_0)\|_{L^2} \,{\rm e}^{-\lambda t} \,,\mbox{ for every }t\in [T,2T-\epsilon]\,.
$$ 
Now 
$$
\begin{aligned}
\|\varphi_{t}-\bar\varphi\|_{H^{n}}=&\left\|\int_{T}^{t}Q(\varphi_\tau)\,d\tau+\varphi_{T}-\bar\varphi \right\|_{H^{n}}
\leq \int_{T}^{t}\|Q(\varphi_\tau)\|_{H^{n}}\,d\tau+\|\varphi_{T}-\bar\varphi\|_{H^{n}}\\
&\leq C\|Q(\varphi_0)\|_{L^2}\frac{{\rm e}^{-\lambda T}}{\lambda}+\|\varphi_{T}-\bar\varphi\|_{H^{n}}\\
&\leq C\|Q(\varphi_0)\|_{L^2}\left(\frac{{\rm e}^{-\lambda T}}{\lambda}+\frac{{\rm e}^{-\lambda\epsilon}}{\lambda}\right)+\|\varphi_{\epsilon}-\bar\varphi\|_{H^{n}}\leq \delta'\,,\quad t\in [T,2T-\epsilon]
\end{aligned}
$$
therefore the flow can be extended in $M\times [0,3T-2\epsilon ]$ with exponential decay in $[2T-\epsilon,3T-2\epsilon]$.
Analogously 
$$
\begin{aligned}
\|\varphi_{t}-\bar\varphi\|_{H^{n}}=&\left\|\int_{2T-\epsilon}^{t}Q(\varphi_\tau)\,d\tau+\varphi_{2T-\epsilon}-\bar\varphi \right\|_{H^{n}}
\leq \int_{2T-\epsilon}^{t}\|Q(\varphi_\tau)\|_{H^{n}}\,d\tau+\|\varphi_{2T-\epsilon}-\bar\varphi\|_{H^{n}}\\
&\leq C\|Q(\varphi_0)\|_{L^2}\frac{{\rm e}^{-\lambda (2T-\epsilon)}}{\lambda}+\|\varphi_{2T-\epsilon}-\bar\varphi\|_{H^{n}}\\
&\leq C\|Q(\varphi_0)\|_{L^2}\left(\frac{{\rm e}^{-\lambda (2T-\epsilon)}}{\lambda}+\frac{{\rm e}^{-\lambda T}}{\lambda}+\frac{{\rm e}^{-\lambda\epsilon}}{\lambda}\right)+\|\varphi_{\epsilon}-\bar\varphi\|_{H^{n}}\leq \delta'\,,
\end{aligned}
$$
for $t\in [2T-\epsilon,3T-2\epsilon]$ and the flow can be extended in $M\times [0,4T-3\epsilon]$ with exponential decay in $[3T-2\epsilon,4T-3\epsilon]$. 
In this way for any $t\in [NT-(N-1)\epsilon,(N+1)T-N\epsilon]$ we have 
$$
\|\varphi_t-\bar \varphi\|_{H^{n}}\leq C\|Q(\varphi_0)\|_{L^2}\frac{{\rm e}^{-\lambda \epsilon}}{\lambda}\sum_{j=0}^{N}{\rm e}^{-\lambda j(T-\epsilon)}+\|\varphi_{\epsilon}-\bar\varphi\|_{H^{n}}\leq \delta'
$$
and the solution $\varphi$ is defined in $M\times [0,\infty)$. Now let $\varphi_\infty:=\varphi_0+\int_{0}^\infty Q(\varphi_s)ds\in  C^{\infty}(M,E)$;
since 
$$
\lim_{t\to \infty}\|\varphi_t-\varphi_{\infty}\|_{H^n}\leq \lim_{t\to \infty}C\|Q(\varphi_0)\|_{L^2}{\rm e}^{-\lambda t}=0\,, \mbox{ for $n$ large enough} 
$$
$\varphi_t$ converges to $\varphi_\infty$ in $C^\infty$-topology.  We clearly have $D_+\varphi_\infty=0$, since $D_+\varphi_t=0$ for every $t\in [0,\infty)$ and by 
construction 
$$
\|\varphi_t-\bar\varphi\|_{C^0}\leq C'\delta', \mbox{ for every }t\in [0,\infty)\,, 
$$
where $C'$ does not depend on $\delta$.  
So up to take $\delta'$ smaller we have $\varphi_\infty\in C^{\infty}(M,\mathcal E)$. Finally 
$$
Q(\varphi_\infty)=\lim_{t\to 0} Q(\varphi_t)=0
$$
and the claim follows. 
\end{proof}



\end{document}